\documentclass[12pt]{article}

\usepackage{diagbox}
\usepackage{mathtools}
\usepackage{bbm}
\usepackage{latexsym}
\usepackage{epsfig}
\usepackage{amsmath,amsthm,amssymb,enumerate}
\usepackage[hidelinks]{hyperref}
\usepackage[active]{srcltx}
\usepackage{color}
\def\red{\color{black}}
\def\blue{\color{black}}

\newcounter{rot}%\addtocounter{rot}{1}, \therot

\def\a{\alpha} \def\b{\beta}  
\def\e{\varepsilon} \def\f{\phi}   
\def\G{\Gamma}  
     \def\l{\lambda}

 \def\om{\omega}  \def\U{\Upsilon}

\newtheorem{theorem}{Theorem}
\newtheorem{lemma}[theorem]{Lemma}
\newtheorem{corollary}[theorem]{Corollary}

\newcommand{\hW}{\widehat{W}}
\newcommand{\hC}{\widehat{C}}

\def\Pr{\mbox{{\bf Pr}}}

\newcommand{\ignore}[1]{}

\newcommand{\cA}{{\cal A}}

\newcommand{\beq}[2]{\begin{equation}\label{#1}#2\end{equation}}

\newcommand{\mult}[2]{\begin{multline}\label{#1}#2\end{multline}}
\usepackage{tikz}
\usetikzlibrary{arrows}
\usetikzlibrary{decorations}
\usetikzlibrary{shapes.misc}

\usepackage{amssymb,amsthm,amsmath}
\usepackage{enumerate}
\usepackage{graphicx,color}

\newcommand{\dd}{\mathrm{d}}
\newcommand{\E}{\mathbb{E}}
\newcommand{\1}{\textbf{1}}

\newcommand{\p}[1]{\mathbb{P}\left( #1 \right)}

\newcommand{\set}[1]{\left\{#1\right\}}
\newcommand{\brac}[1]{\left(#1\right)}

\newcommand{\bfrac}[2]{\left(\frac{#1}{#2}\right)}

\usepackage[paper=a4paper, left=1.4in, right=1.4in, top=1.5in, bottom=1.5in]{geometry}
\title{\vspace{-3em}A randomly weighted minimum arborescence with a random cost constraint}
\author{Alan Frieze\thanks{Research supported in part by NSF grant DMS1661063}
\ and 
Tomasz Tkocz\\
Carnegie Mellon University\\Pittsburgh PA15213\\U.S.A.
}
\date{}
\def\l{\lambda}

\begin{document}

\maketitle

\begin{abstract}
We study the minimum spanning arborescence problem on the complete digraph $\vec{K}_n$ where an edge $e$ has a weight $W_e$ and a cost $C_e$, each of which is an independent uniform random variable $U^s$ where $0<s\leq 1$ and $U$ is uniform $[0,1]$. There is also a constraint that the spanning arborescence $T$ must satisfy $C(T)\leq c_0$. We establish, for a range of values for $c_0,s$, the asymptotic value of the optimum weight via the consideration of a dual problem. 
\end{abstract}

\bigskip

\begin{footnotesize}

\end{footnotesize}
\section{Introduction}
Let $U$ denoe the uniform $[0,1]$ random variable and let $0<s\leq 1$. We consider the minimum spanning arborescence problem in the context of the complete digraph $\vec{K}_n$ where each edge has an independent copy of $U^s$ for weight $W_e$ and an independent copy of $U^s$ for cost $C_e$. Let $\cA$ denote the set of spanning arborescences of $\vec{K}_n$. An arborescence is a rooted tree in which every edge is directed away from the root. The weight of a spanning arborescence $A$ is given by $W(A)=\sum_{e\in A}W_e$ and its cost $C(A)$ is given by $C(A)=\sum_{e\in A}C_e$. The problem we study is
\beq{prob}{
\text{Minimise }W(A)\text{ subject to }A\in\cA,\,C(A)\leq c_0,
}
where $c_0$ may depend on $n$. 

Without the constraint $C(A)\leq c_0$, we have a weighted matroid intersection problem and as such it is solvable in polynomial time, see for example Lawler \cite{Law}. Furthermore Edmonds \cite{Ed} gave a particularly elegant algorithm for solving this problem. With the constraint $C(A)\leq c_0$, the problem becomes NP-hard, since the knapsack problem can be easily reduced to it. On the other hand, equation \eqref{prob} defines a natural problem that has been considered in the literature, in the worst-case rather than the average case. See for example Guignard and Rosenwein \cite{GR} and Aggarwal, Aneja and Nair \cite{AAN} and Goemans and Ravi \cite{GRa} (for an undirected version). This paper is a follow up to the analysis of the cost constrained minimum weight spanning tree problem considered in \cite{FT}.

The addition of a cost contraint makes the problem NP-hard and reflects the fact that in many practical situations there may be more than one objective for an optimization problem. Here the goal is to lower weight and cost. We first consider the easier case where $s=1$.

{\red We use the following notation. For two real sequences $A_n$ and $B_n$ we say that $A_n\approx B_n$ if there exists a function $\e=\e(n) \to 0$ as $n \to \infty$ such that for every $n$
\[
(1-\e)A_n\leq B_n\leq (1+\e)A_n.
\]

}

\begin{theorem}\label{thm:main-res}
Let $D_n$ be the complete digraph $\vec{K}_n$ on $n$ vertices with each edge $e$ having assigned a random weight $W_e$ and a random cost $C_e$, where $\{W_e, C_e\}$ is a family of i.i.d. random variables uniform on $[0,1]$. Given $c_0 > 0$, let $W^*_{arb}$ be the optimum value for the problem \eqref{prob}. The following hold w.h.p.
\begin{enumerate}[{\bf C{a}se 1:}]
\item %If $c_0 \in \sqrt{\frac{\pi}{8}}[\sqrt{\log n},\frac{n}{(\log n)^2}]$, then 
{\red If $c_0 \in [\om,\frac{n}{\om}]$ where $\om\to\infty$ then}
\beq{optb}{
W^*_{arb} \approx \frac{\pi n}{8 c_0}.
}
\item Suppose now that $c_0 = \alpha n$, where $\a=O(1)$ is a positive constant. 
\begin{enumerate}[(i)]
\item If $\a>1/2$ then 
\[
W^*_{arb}\approx 1.
\]
\item If $\a< 1/2$ then
\[
W^*_{arb}\approx f(\b^*)-\a\b^*
\] 
where $\b^*$ is the unique positive solution to $f'(\b)=\a$ and where
 \[
f(\b)=\b^{1/2}\int_{t=0}^{\b^{1/2}}e^{-t^2/2}\dd t +e^{-\b/2}, \qquad \beta > 0.
\]
\end{enumerate}
\item Suppose now that $c_0 = \alpha$, where $\a=O(1)$ is a positive constant. 
\begin{enumerate}[(i)]
\item If $\a<1$ then there is no solution to \eqref{prob}.
\item If $\a>1$ then 
\[
W^*_{arb}\approx (g(\b)-\a\b)n
\] 
where $\b^*$ is the unique positive solution to $g'(\b)=\a$ and where
 \begin{align*}
g(\b)&=\b^{1/2}\int_{t=0}^{\b^{-1/2}}e^{-t^2/2}\dd t +\b e^{-1/2\b}\\
&=\b f(1/\b), \qquad\qquad\qquad\qquad \beta > 0.
\end{align*}
\end{enumerate}
\end{enumerate}
\end{theorem}
{\red Some simple observations: if $\a>1/2$ in Case 2 then w.h.p. cost of the arborescence of minimum total weight satisfies the cost contraint and we will see that $W^*_{arb}\approx 1$ in this case. If $\a<1$ in Case 3 then the problem is infeasible. The expected minimum weight of an edge leaving a fixed vertex is $1/n$ and summing these gives a lower bound on the minimum weight of an arborescence. 

It is also instructive to examine Case 2 when $\a\to0$ and Case 3 when $\a\to\infty$ to see if we ``recover'' Case 1. First consider Case 2. Lemma \ref{f(b)} below shows that $f'(\b)$ decreases monotonically to zero which means that $\b^*\to\infty$ as $\a\to 0$. The lemma also shows that for large $\b$ we have $\a=f'(\b)\approx (\pi/8\b)^{1/2}$. We have $f(\b)\approx (\pi\b/2)^{1/2}$ and so $f(\b)-\a\b\approx (\pi\b/8)^{1/2}\approx (\pi/8\a)$, and this is consistent with Case 1. For Case 3, Lemma \ref{g(b)} implies that $\b\to0$ as $\a\to\infty$. In which case, $\a=g'(\b)\approx (\pi/8\b)^{1/2}$ and $g(\b)\approx (\pi\b/2)^{1/2}$ and we again get an expression consistent with Case 1.
}

We note that Lemma \ref{f(b)} of Section \ref{fg} shows that the claims in Case 2 are reasonable and Lemma \ref{g(b)} shows that the claims in Case 3 are reasonable (that is, the stated equations possses unique solutions).

For the case $s<1$ we will prove the following.
\begin{theorem}\label{thm:main-resx}
Let $D_n$ be the complete digraph $\vec{K}_n$ on $n$ vertices with each edge $e$ being assigned a random weight $W_e$ and a random cost $C_e$, where $\{W_e, C_e\}$ is a family of i.i.d. random variables $U^s$. Given 
\beq{c0range}{
n^{1-s}\log n\ll c_0\ll \frac{n}{\log n},
}
let $W^*_{arb}$ be the optimum value for the problem \eqref{prob}. The following holds w.h.p.
\beq{opta}{
W^*_{arb}\approx \frac{C_s^2n^{2-s}}{4c_0},
}
where 
\[
C_s={\blue \G(s/2+1)}\bfrac{\G(2/s+1)}{\G(1/s+1)^2}^{s/2}.
\]
\end{theorem}
Note that $\G(3/2)=\sqrt{\pi}/2$ and this implies that $C_1=\sqrt{\pi/2}$ and the expression in \eqref{opta} is consistent with the expression in \eqref{optb}.

We will first concentrate on the case $s=1$. After this, we will continue with the proof of Theorem \ref{thm:main-resx}. We note that a preliminary version containing the results for the case $s=1$ appeared in \cite{SODA}. The weights and costs will therefore be uniform $[0,1]$ until we reach the more general case in Section \ref{a<1}. We will then prove Theorem \ref{thm:main-resx} as stated and then show how to extend this result to a wider class of distribution via a simple coupling argument from Janson \cite{Jan}.

{\red
\section{Outline of the argument}
We first prove an equivalent result where we replace spanning arborescence by a surrogate, a functional digraph. Given $f:[n]\to[n]$ we let the associated digraph $D_f$ be $([n],A_f)$ where $A_f=\set{(v,f(v)):v\in [n]}$. We will estimate the minimum weight of the set of edges in a $D_f$ that satisfies $C(A_f)\leq c_0$. Computing this minimum is a 0-1 integer program $IP_{map}$ and we estimate the value of the LP relaxation $LP_{map}$ and show that w.h.p. $IP_{map}$ and $LP_{map}$ have asymptotically equal minimum objective values. A random function $f$ gives rise to a $D_f$ that is close enough to being an arborescence that we can translate results for the optimal $f$ to results for the optimal arborescence.

We get our estimate for the optimum objective value in $LP_{map}$ by using Lagrangean Relaxation. We used the same approach in \cite{FT} for the constrained spanning tree problem. The main difference there is that in the case of spannning trees we can estimate the dual value directly via an integral formula. We do not need a surrogate as we did here. So the philosophy is the same, but the details differ substantially. 

Lagrangean Relaxation introduces a dual function $\f(\l)$, where $\l$ is the ``Lagrange Multiplier''. We have to maximise $\f$ and while it is straightforward to estimate $\E\f$, to maximise $\f$ we need concentration around the mean. This is the subject of Section \ref{concsec}. Rather than attempt a union bound over all real $\l$, we discretize the set of values for $\l$ and show we have sufficient concentration to get a very good estimate for the maximum. This is done in Section \ref{discrete}. The optimum solution to $LP_{map}$ is almost a mapping and it needs to be converted to an actual mapping.

So, the proof goes: 
\begin{enumerate}
\item Define the dual problem for finding a minimum weight functional digraph and get an expression for $\E(\f(\l))$.
\item Prove concentration of $\f$ around its mean.
\item Discretize and apply concentration and find the maximum value of $\f$.
\item Convert the optimum LP solution to a {\em random mapping} that is close in weight and cost to the optimum LP solution.
\item Argue that the duality gap for the functional digraph problem is small.
\item Transfer the result on functional digraphs to spanning arborescences.
\end{enumerate}
}
\section{Auxiliary results}

\subsection{Properties of the functions $f$ and $g$}\label{fg}
\begin{lemma}\label{f(b)}
$f(0)=1, f(\infty)=\infty$, $f'(0)=1/2, f'(\infty) = 0$ and $f'$ is strictly monotone decreasing. These imply that $f' > 0$, $f$ is concave increasing and for every $0 < \alpha < \frac{1}{2}$, there is a unique $\beta > 0$ such that $f'(\beta) = \alpha$.
\end{lemma}
\begin{proof}
This follows by inspection of $f$ and 
\begin{align*}
f'(\b)&=\frac{1}{2\b^{1/2}}\int_{t=0}^{\b^{1/2}}e^{-t^2/2}\dd t.\\
f''(\b)&=\frac{1}{4\b^{3/2}}\int_{t=0}^{\b^{1/2}}\brac{e^{-\b/2}-e^{-t^2/2}}\dd t<0.
\end{align*}
\end{proof}
\begin{lemma}\label{g(b)}
$g'(0)=\infty, g'(\infty)=1$ and $g'$ is strictly monotone decreasing. This implies that $g$ is concave and for every $\alpha > 1$, there is a unique $\beta > 0$ such that $g'(\beta) = \alpha$.
\end{lemma}
\begin{proof}
We have $g(\b) = \b f( 1/\b)$ and
\begin{align*}
g'(\b)&=f(1/\b)-\frac{1}{\b}f'(1/\b) = \frac{1}{2\b^{1/2}}\int_{t=0}^{\b^{-1/2}}e^{-t^2/2}\dd t +e^{-1/2\b}.\\
g''(\b)&=\frac{1}{\b^{3}}f''(1/\b)<0.
\end{align*}
By inspection, $g'(0) = \infty$ and $g'(\infty) = 1$.
\end{proof}
\subsection{Expectation}
Our strategy will be to prove results about mappings $f:[n]\to [n]$, where $f(i)\neq i,i\in [n]$. Given $f$, we have a digraph $D_f$ with vertex set $[n]$ and edge set $A_f=\set{(i,f(i)):i\in [n]}$. Most of the analysis concerns the problem\\
{\bf Minimum Weight Constrained Mapping (MWCM):}
\[
\text{Minimise }W_{map}(f)=\sum_{i\in [n]}W_{(i,f(i))}\text{ subject to }C(f)=\sum_{i\in [n]}C_{(i,f(i))}\leq c_0.
\]
Let $f^*$ solve MWCM. We will argue that w.h.p. $D_{f^*}$ is close to being an arborescence and that a small change will result in a near optimum arborescence that will verify the claims of Theorem \ref{thm:main-res}. The following lemma begins our analysis of optimal mappings. {\red To motive it, note that the minimum over the mappings in the Lagrangean dual of MWCM is simply attained by choosing the best edge for each vertex (see \eqref{eq:phi_m} below).} We have expressed the following calculations with $n$ replacing $n-1$, but this does not affect the final results.
\begin{lemma}\label{lm:Emin}
Let $X_1,X_2,\ldots$ and $Y_1,Y_2,\ldots$ be i.i.d. random variables uniform on $[0,1]$. Then
\begin{enumerate}[{\bf E1:}]
\item For $\lambda \leq \frac{1}{n\log n}$, we have
\begin{equation}\label{eq:Emin-lam-small-very}
\E \min_{i \leq n}\{X_i + \lambda Y_i\} =(1+o(1))\frac{1}{n}.
\end{equation}
\item For $\frac{1}{n\log n} \leq \lambda \leq \frac{\log n}{n}$, we have
\begin{equation}\label{eq:Emin-lam-small-mid}
\E \min_{i \leq n}\{X_i + \lambda Y_i\} =(1+o(1))\frac{1}{n}\left(\sqrt{\lambda n}\int_0^{\sqrt{\lambda n}} e^{-\frac{t^2}{2}} \dd t + e^{-\lambda n/2}\right).
\end{equation}
%{\red E2 and E3 are the same if $\l=\om/n$.}
\item For $\frac{\log n}{n} \leq \lambda \leq \frac{n}{\log n}$, we have
\begin{equation}\label{eq:Emin-lam-small}
\E \min_{i \leq n}\{X_i + \lambda Y_i\} =(1+o(1))\sqrt{\frac{\pi}{2}}\sqrt{\frac{\lambda}{n}}.
\end{equation}
\item For $\frac{n}{\log n} \leq \lambda \leq n\log n$, we have
\begin{equation}\label{eq:Emin-lam-mid}
\E \min_{i \leq n}\{X_i + \lambda Y_i\} =(1+o(1))\frac{\lambda}{n}\left(\sqrt{\frac{n}{\lambda}}\int_0^{\sqrt{\frac{n}{\lambda}}} e^{-\frac{t^2}{2}} \dd t + e^{-\frac{1}{2}\frac{n}{\lambda}}\right).
\end{equation}
%{\red E3 and E4 are the same if $\l=n/\om$.}
\item For $\lambda \geq n\log n$, we have
\begin{equation}\label{eq:Emin-lam-large}
\E \min_{i \leq n}\{X_i + \lambda Y_i\} =(1+o(1))\frac{\lambda}{n}.
\end{equation}
\end{enumerate}
\end{lemma}
\begin{proof}
Thanks to independence
\begin{align}
\E \min_{i \leq n}\{X_i + \lambda Y_i\} &= \int_0^{\infty} \p{\min_{i \leq n}\{X_i + \lambda Y_i\} > t} \dd t \nonumber\\
&= \int_0^{\infty} \Big[\p{X_1 + \lambda Y_1 > t}\Big]^n \dd t.\label{tti}
\end{align}
\textbf{Case 1.} $\lambda \geq 1$. \\
It follows from an elementary computation that (for details see e.g. the appendix in \cite{FT})
\[
\p{X_1 + \lambda Y_1 > t} = \begin{cases}
1 - \frac{t^2}{2\lambda}, & 0 < t < 1, \\
1 + \frac{1}{2\lambda}-\frac{t}{\lambda}, & 1 \leq t < \lambda, \\
\frac{(1+\lambda-t)^2}{2\lambda}, & \lambda \leq t < 1+\lambda, \\
0, & t \geq 1+\lambda.
\end{cases}
\]
Thus,
\begin{align}\notag
\E \min_{i \leq n}\{X_i + \lambda Y_i\} &= \int_0^1 \left(1 - \frac{t^2}{2\lambda}\right)^n \dd t \\\notag
&+ \int_1^{\lambda} \left(1 + \frac{1}{2\lambda}-\frac{t}{\lambda}\right)^n \dd t \\\notag
&+ \int_\lambda^{1+\lambda} \left(\frac{(1+\lambda-t)^2}{2\lambda}\right)^n \dd t \\\label{eq:Emin-lam>1}
&= \int_0^1 \left(1 - \frac{t^2}{2\lambda}\right)^n \dd t \\\notag
&\quad+ \frac{\lambda}{n+1}\left[\left(1 - \frac{1}{2\lambda}\right)^{n+1}-\left(\frac{1}{2\lambda}\right)^{n+1}\right] + \frac{1}{2n+1}\left(\frac{1}{2\lambda}\right)^{2n}.
\end{align}

\bigskip
\textbf{Case 1.1.} $1 \leq \lambda \leq \frac{n}{\log n}$\\
A change of variables gives
\begin{equation}\label{eq:integral}
\int_0^1 \left(1 - \frac{t^2}{2\lambda}\right)^n \dd t = \sqrt{\lambda}\int_0^{\frac{1}{\sqrt{\lambda}}} \left(1 - \frac{t^2}{2}\right)^n \dd t.
\end{equation}
We have {\red $\sqrt{\log n/n} \leq \frac{1}{\sqrt{\lambda}} \leq 1$} and
\[
\int_{\sqrt{\log n/n}}^{\red 1} \left(1 - \frac{t^2}{2}\right)^n \dd t \leq \int_{\sqrt{\log n/n}}^1 e^{-\frac{nt^2}{2}} \dd t = \frac{1}{\sqrt{n}}\int_{\sqrt{\log n}}^1 e^{-\frac{t^2}{2}} \dd t = o(n^{-1/2}).
\]
Therefore
\[
\sqrt{\lambda}\int_0^{\frac{1}{\sqrt{\lambda}}} \left(1 - \frac{t^2}{2}\right)^n \dd t = \sqrt{\lambda}\int_0^{\sqrt{\log n/n}} \left(1 - \frac{t^2}{2}\right)^n \dd t + \sqrt{\lambda}o(n^{-1/2}).
\]
Using $1+x = e^{x + O(x^2)}$ as $x \to 0$, we get
\begin{align*}
\int_0^{\sqrt{\log n/n}} \left(1 - \frac{t^2}{2}\right)^n \dd t &= \int_0^{\sqrt{\log n/n}} e^{-\frac{nt^2}{2}+O(nt^4)} \dd t \\
&= (1+o(1))\int_0^{\sqrt{\log n/n}} e^{-\frac{nt^2}{2}} \dd t \\
&= (1+o(1))\frac{1}{\sqrt{n}}\int_0^{\sqrt{\log n}} e^{-\frac{t^2}{2}} \dd t \\
&= (1+o(1))\frac{1}{\sqrt{n}}\int_0^\infty e^{-\frac{t^2}{2}} \dd t + o(n^{-1/2}) \\
&= (1+o(1))\frac{1}{\sqrt{n}}\sqrt{\frac{\pi}{2}} + o(n^{-1/2}).
\end{align*}
Putting these together back into \eqref{eq:integral} yields
\[
\int_0^1 \left(1 - \frac{t^2}{2\lambda}\right)^n \dd t = (1+o(1))\sqrt{\frac{\pi}{2}}\sqrt{\frac{\lambda}{n}} + \sqrt{\lambda}o(n^{-1/2}) = (1+o(1))\sqrt{\frac{\pi}{2}}\sqrt{\frac{\lambda}{n}}.
\]
Since
\begin{align*}
\frac{\lambda}{n+1}\left[\left(1 - \frac{1}{2\lambda}\right)^{n+1}-\left(\frac{1}{2\lambda}\right)^{n+1}\right] + \frac{1}{2n+1}\left(\frac{1}{2\lambda}\right)^{2n} &= O\left(\frac{\lambda}{n}\right) \\
&= \sqrt{\frac{\lambda}{n}}O\left(\sqrt{\frac{1}{\log n}}\right),
\end{align*}
from \eqref{eq:Emin-lam>1} we can finally obtain \eqref{eq:Emin-lam-small}.

\bigskip
\textbf{Case 1.2.} $\frac{n}{\log n} \leq \lambda \leq n\log n$\\
Since for $t \leq \frac{1}{\sqrt{\lambda}}$, $(1 - \frac{t^2}{2})^n =  e^{-\frac{nt^2}{2}}e^{O(nt^4)} = e^{-\frac{nt^2}{2}}e^{O(\frac{\log^2 n}{n})}$, directly from \eqref{eq:integral}, we get
\[
\int_0^1 \left(1 - \frac{t^2}{2\lambda}\right)^n \dd t = (1+o(1))\sqrt{\lambda}\int_0^{\frac{1}{\sqrt{\lambda}}} e^{-\frac{nt^2}{2}} \dd t = (1+o(1))\sqrt{\frac{\lambda}{n}}\int_0^{\sqrt{\frac{n}{\lambda}}} e^{-\frac{t^2}{2}} \dd t.
\]
Moreover,
\begin{align*}
\frac{\lambda}{n+1}&\left[\left(1 - \frac{1}{2\lambda}\right)^{n+1}-\left(\frac{1}{2\lambda}\right)^{n+1}\right] + \frac{1}{2n+1}\left(\frac{1}{2\lambda}\right)^{2n} \\
&=(1+o(1))\frac{\lambda}{n}e^{-\frac{n}{2\lambda}+ O\left(\frac{n}{\lambda^2}\right)} + O\left(\left(\frac{\log n}{n}\right)^n\right) \\
&=\frac{\lambda}{n}e^{-\frac{n}{2\lambda}}\left(1+o(1) + \frac{n}{\lambda}e^{\frac{n}{2\lambda}}O\left(\left(\frac{\log n}{n}\right)^n\right)\right) \\
&= \frac{\lambda}{n}e^{-\frac{n}{2\lambda}}\left(1+o(1)\right).
\end{align*}
Plugging these back in \eqref{eq:Emin-lam>1} yields \eqref{eq:Emin-lam-mid}.

\bigskip
\textbf{Case 1.3.} $\lambda \geq n\log n$\\
Plainly,
\[
\int_0^1 \left(1 - \frac{t^2}{2\lambda}\right)^n \dd t = O(1) = \frac{\lambda}{n}o(1).
\]
Since $\left(1 - \frac{1}{2\lambda}\right)^{n+1} = e^{O(\frac{n}{\lambda})} = 1 + o(1)$, we have
\begin{align*}
\frac{\lambda}{n+1}&\left[\left(1 - \frac{1}{2\lambda}\right)^{n+1}-\left(\frac{1}{2\lambda}\right)^{n+1}\right] + \frac{1}{2n+1}\left(\frac{1}{2\lambda}\right)^{2n} = \frac{\lambda}{n}\left(1 + o(1)\right).
\end{align*}
Putting these in \eqref{eq:Emin-lam>1} gives \eqref{eq:Emin-lam-large}.

\textbf{Case 2.} $\lambda \leq 1$\\
We write
\[
\E \min_{i \leq n}\{X_i + \lambda Y_i\} = \lambda \E \min_{i \leq n}\{X_i + \lambda^{-1} Y_i\}
\]
and then apply \eqref{eq:Emin-lam-small}, \eqref{eq:Emin-lam-mid} and \eqref{eq:Emin-lam-large} to $\lambda^{-1}$, multiply the answers by $\lambda$ to get \eqref{eq:Emin-lam-small-very}, \eqref{eq:Emin-lam-small-mid} and the missing range $\frac{\log n}{n} \leq \lambda \leq 1$ of \eqref{eq:Emin-lam-small}.
\end{proof}

\begin{corollary}\label{cor:Emin-low-bd}
Under the assumptions of Lemma \ref{lm:Emin}, we have
\beq{corbd}{
n\E \min_{i \leq n}\{X_i + \lambda Y_i\} = \Omega(\max\{1,\sqrt{\lambda n}\}).
}
\end{corollary}
\begin{proof}
This follows directly from \eqref{eq:Emin-lam-small-very} - \eqref{eq:Emin-lam-large} and the fact that $f(\beta) \geq 1$ (Lemma~\ref{f(b)}) as well as the lower bound
\begin{align*}
f(\beta) \geq \max\{\sqrt{\beta}\int_0^{\sqrt{\beta}} e^{-t^2/2}\dd t, e^{-\beta/2}\} &\geq  \max\left\{\sqrt{\beta}\int_0^{\sqrt{\beta}} e^{-t^2/2}\dd t, \sqrt{\beta}\1_{\{\beta \leq \frac{1}{2}\}}\right\} \\
&\geq \frac{1}{2}\sqrt{\beta}.
\end{align*}
\end{proof}

\subsection{Concentration}\label{concsec}
Again $n$ replaces $n-1$ in the calculations.
\begin{lemma}\label{lm:conc}
Let $W_{(i,j)}$ and $C_{(i,j)}$, $i, j \leq n$, be i.i.d. random variables uniform on $[0,1]$. 
Let %$\lambda \in [\frac{\log n}{n}, \frac{n}{\log n}]$. 
$\lambda \in [0, n\log n]$. For $X_i = \min_{j} \{W_{(i,j)} + \lambda C_{(i,j)}\}$, $S = \sum_{i \leq n} X_i$ and $\varepsilon = {\red \Theta(n^{-1/5})}$, we have
\begin{equation}\label{eq:conc}
\p{|S - \E S| > \varepsilon \E S} = O(n^{-99}).
\end{equation}
Moreover,
\begin{equation}\label{eq:conc-Cmax}
\p{\exists i:\; X_i > 10(1+\lambda)\sqrt{\log n/n}} \leq n^{-99}.
\end{equation}
\end{lemma}
\begin{proof}
Let $M = 10(1+\lambda)\sqrt{\log n/n}$ and $B$ be the event that for some $i$, $X_i \geq M$. We have,
\begin{equation}\label{eq:conc-two-terms}
\p{|S - \E S| > \varepsilon \E S} \leq \p{B} + \p{(|S - \E S| > \varepsilon \E S)\wedge B^c}.  
\end{equation}
First we bound $\p{B}$. By the union bound and independence,
\[
\p{B} \leq n\p{X_1 \geq M} = n\left[\p{W_{(1,1)}+\lambda C_{(1,1)} \geq M}\right]^{n}.
\]
We use $W_{(1,1)}+\lambda C_{(1,1)} \leq (1+\lambda)\max\{W_{(1,1)},C_{(1,1)}\}$ and note that since these variables are uniform, we have $\p{\max\{W_{(1,1)},C_{(1,1)}\} \geq u} = 1 - u^2$ for $u < 1$. We thus get
\[
\p{B} \leq n\left[1 - 100\frac{\log n}{n}\right]^{n} \leq ne^{-100\log n} = n^{-99},
\]
which establishes \eqref{eq:conc-Cmax}.

The second term in \eqref{eq:conc-two-terms} can be bounded using Chernoff's inequality because on $B^c$, $X_i = X_i\1_{X_i \leq M}$, that is $S$ can be treated as a sum of $n$ independent random variables $\tilde X_i = X_i\1_{X_i \leq M}$ with $\tilde X_i \in [0,M]$. Clearly $\tilde X_i \leq X_i$ and $\tilde S = \sum \tilde X_i \leq S$, so
\begin{align*}
\p{(|S - \E S| > \varepsilon \E S)\wedge B^c} = \p{(|\tilde S - \E S| > \varepsilon \E S)\wedge B^c}  \leq \p{|\tilde S - \E S| > \varepsilon \E S}.
\end{align*}
By the Chernoff bound
\[
\p{|\tilde S - \E \tilde S| > \varepsilon \E \tilde S} \leq 2\exp\left\{-\frac{\varepsilon^2\E\tilde S}{3M}\right\}.
\]
Note that
\[
|\tilde S - \E S| \leq |\tilde S - \E \tilde S| + |\E S - \E \tilde S|.
\]
and
\begin{align*}
 |\E S - \E \tilde S| = \left|\E \sum X_i\1_{X_i > M} \right| \leq (1+\lambda)\E\sum \1_{X_i > M} &\leq (1+\lambda)n\p{X_1  > M} \\
 &= {\red O((n\log n)\cdot n\cdot n^{-99}) = O(n^{-97})},
\end{align*}
thanks to \eqref{eq:conc-Cmax}. Moreover, 
%By \eqref{eq:Emin-lam-small},
by Corollary \ref{cor:Emin-low-bd},
\[
\E S = \Omega(\max\{1,\sqrt{\lambda n}\})
\]
and  {\red by the assumption $\e = \Omega(n^{-1/5})$}, so
\[
|\E S - \E \tilde S| {\red = O(n^{-97}) \leq \Omega(n^{-1/5})} \leq \frac{1}{2}\varepsilon \E S
\]
and we get
\begin{align*}
\p{|\tilde S - \E S| > \varepsilon \E S} \leq  \p{|\tilde S - \E \tilde S| > \frac{1}{2}\varepsilon \E S} &\leq \p{|\tilde S - \E \tilde S| > \frac{1}{2}\varepsilon \E \tilde S} \\
&\leq 2\exp\left\{-\frac{\varepsilon^2\E\tilde S}{12M}\right\}.
\end{align*}
Finally, observe that
\[
\frac{\E \tilde S}{M} \geq \frac{\E S}{2M} = \frac{\Omega(\max\set{1,\sqrt{\lambda n}})}{20(1+\lambda)\sqrt{\log n}}\sqrt{n}
\]
and for $\lambda \leq n\log n$, we have $\frac{\max\{1,\sqrt{\lambda n}\}}{1+\lambda} \geq \frac{1}{2}\sqrt{\frac{1}{\log n}}$. Consequently,
\[
\frac{\varepsilon^2\E \tilde S}{12M} = \Omega\left(\frac{\varepsilon^2 \sqrt{n}}{\log n}\right) = {\red \Omega(n^{-1/10}n^{1/2}/\log n)} = \Omega(n^{1/10}),
\]
so
\[
%\begin{equation}\label{eq:conc-B}
\p{|S - \E S| > \varepsilon \E S, B^c} = O(e^{-n^{1/10}}).
\]
%\end{equation}
In view of \eqref{eq:conc-two-terms}, this combined with \eqref{eq:conc-Cmax} finishes the proof of \eqref{eq:conc}.
\end{proof}

\begin{corollary}\label{cor1}
Let $M_n$ denote the minimum  weight of a mapping with weights $W_e+\l C_e,e\in E(\vec{K}_n)$. Then with probability $1-O(n^{-90})$,
\[
M_n\approx \begin{cases}(\pi\l n/2)^{1/2}&{\bf E3/Case\ 1}.\\f(\l n)&{\bf E2/Case\ 2}.\\ng(\l/n) &{\bf E4/Case\ 3}.\end{cases}
\]
\beq{Wmax}{
W_{\max}{\red =\max_i W_i}\leq \begin{cases}O\brac{(1+\l)\sqrt{\log n/n}}&{\bf E3/Case\ 1}.\\O\brac{\sqrt{\log n/n}}&{\bf E2/Case\ 2}.\\1&{\bf E4/Case\ 3}.\end{cases}
}
\beq{Cmax}{
C_{\max}{\red =\max_iC_i}\leq \begin{cases}O\brac{\frac{1}{\l}+1}\sqrt{\log n/n}&{\bf E3/Case\ 1}.\\1&{\bf E2/Case\ 2}.\\O(\log n/n)&{\bf E4/Case\ 3}.\end{cases}
}
\end{corollary}
\begin{proof}
The claim about $M_n$ follows directly from Lemma \ref{lm:Emin} and Lemma \ref{lm:conc}.\\
For Cases 1 and 2 the claim about $W_{\max}$ follows from \eqref{eq:conc-Cmax} and the claim for Case 3 is trivial.\\ 
For Case 1 the claim about $C_{\max}$ follows from \eqref{eq:conc-Cmax} and the claim for Case 2 is trivial. For Case 3, we let $p=K\log n/n$ and argue that w.h.p. for each $v\in [n]$, there exists $w\neq v$ such that $C_{(v,w)}\leq p$ (the probability of the contrary is at most $n(1-p)^{n-1}=o(1)$). If $C_{\max}=C_{(v_1,w_1)}> 2p$ then replacing $(v_1,w_1)$ by $(v_1,w_2)$ where $C_{(v_1,w_2)}\leq p$ we reduce the value $W(F)+\l C(F)$ of the supposed mapping $F$, by at least $\l p-1 \geq \frac{n}{\log n} K \frac{\log n}{n} -1 > 0$, contradicting the optimality of $F$.
\end{proof}

\subsection{Properties of optimal dual solutions}\label{Dual}
Let 
\[
I=\set{(i,j)\in [n]^2:i\neq j\text{ and }W_{i,j},C_{i,j}\text{ are bounded by \eqref{Wmax}, \eqref{Cmax} respectively}}.
\]
For $i\in [n]$ we let $J_i=\set{j:(i,j)\in I}$.

We can express the problem MWCM as the following integer program:\\
{\bf $IP_{map}$}
\begin{align}
\text{Minimize }&\sum_{(i,j)\in I}W_{i,j}x_{i,j}\text{ subject to }\nonumber\\
&\sum_{j\in J_i}x_{i,j}=1, i\in [n]\label{IP1}\\
&\sum_{(i,j)\in [I]}C_{i,j}x_{i,j}\leq c_0\label{IP2}\\
&x_{i,j}=0\text{ or }1,\quad\text{ for all }i\neq j.\label{IP3}
\end{align}
We obtain the relaxation $LP_{map}$ by replacing \eqref{IP3} by 
\beq{LP1}{
\text{$0\leq x_{i,j}\leq 1$ for all $(i,j)\in I$.}
}

We will consider the dual problem: we will say that a map $f$ is {\em feasible} if $f(i)\in J_i$ for $i\in [n]$. We let $\Omega^*$ denote the set of feasible $f$.

$Dual_{map}(W,C,c_0)$:
\[
\text{Compute }\max_{\l\geq 0}\phi_{{\red map}}(\l,c_0)\text{ where }
\f_{\red map}(\l,c_0)=\min_{f\in\Omega^*}\set{\sum_{i\in [n]}(W_{i,f(i)}+\l C_{i,f(i)})-\l c_0}.
\]
Now it is well known (see for example \cite{NW}) that 
\[
\max_{\l\geq 0}\phi_{\red map}(\l,c_0)=\min\set{\sum_{(i,j)\in I}W_{i,j}x_{i,j}\text{ subject to } \eqref{IP1}, \eqref{IP2}, \eqref{LP1}}.
\]
I.e. maximising $\f_{\red map}$ solves the linear program $LP_{map}$. The basic feasible solutions to the linear program $LP_{map}$ have a rather simple structure. {\red A basis matrix is obtained by replacing a single row of the $n\times n$ identity matrix $I_n$ with coefficients from the LHS of \eqref{IP2} (or it is $I_{n+1}$). This is because there are exactly $n+1$ basic variables. We either (i) have a single $i^*$ such that (a) $i\neq i^*$ implies that there is a unique $j(i)$ such that $x_{i,j(i)}=1$ and $x_{i,j}=0$ for $j\neq j(i)$ and (b) there are two indices $j_1,j_2$ such that $x_{i^*,j_\ell}\neq 0,\ell=1,2$ or (ii) for each  
$i$ there is a basic $x_{i,j(i)}$ and the $(n+1)$th basic variable is the slack in \eqref{IP2}.} 

We are using Corollary \ref{cor1} to restrict ourselves to feasible $f$, so that we may use the upper bounds in \eqref{Wmax}, \eqref{Cmax}.

Consider the unique (with probability one) basic feasible solution that solves {\red $LP_{map}$}. The optimal shadow price $\l^*$ is also the optimal solution to the dual problem $DUAL_{map}(W,C,c_0)$. Let the map $f^*=f^*(c_0)$ be obtained from an optimal basic feasible solution to $LP_{map}$ by (i) putting $x_{i^*,j_1}=x_{i^*,j_2}=0$ and then (ii) choosing $j^*$ to minimise $C_{i^*,j}+\l^*W_{i^*,j}$ and then putting $x_{i^*,j^*}=1$. This yields the map $f^*$, where $f^*(i)=j(i),i\neq i^*$ and $f^*(i^*)=j^*$.

Let $W_{\max}=\max\set{W_{i,f^*(i)}:i\in [n]}$ and define $C_{\max}$ similarly.  Let $W^*_{LP}$ denote the optimal objective value to $LP_{map}$. Then we clearly have
\beq{clearly}{
W(f^*)\leq W^*_{LP}+W_{\max}\text{ and }C(f^*)\leq c_0+C_{\max}.
}

\begin{lemma}\label{lm:opt-is-unif}
Let $W_{(i,j)}$ and $C_{(i,j)}$, $i, j \leq n$, be i.i.d. random variables on $[0,1]$. Then $f^*$ is distributed as a random mapping.
\end{lemma}
\begin{proof}
Fix $f_0 \in [n]^{[n]}$ and a permutation $\pi$ of $[n]$. The distribution of $f^*$ is invariant with respect to relabelling (permuting) the domain $[n]$, that is $\pi\circ f^*$ and $f^*$ have the same distribution. Therefore,
\[
\p{f^* = f_0} = \p{\pi\circ f^* = \pi\circ f_0} = \p{f^* = \pi\circ f_0}.
\] 
\end{proof}
{\red The importance of the above lemma stms from the fact that it implies that w.h.p. $D_{f^*}$ has $O(\log n)$ components and as such we only need to change $O(\log n)$ edges to make it an arborescence. }
\subsection{Discretisation}\label{discrete}
We divide the interval $[0,n\log n]$ into $n^{10}$ intervals $[\lambda_i,\lambda_{i+1}]$ of equal length. Then $|\lambda_{i+1}-\lambda_i| \leq {\red n^{-8}}$. By standard arguments we have the following claim about the maximum after the discretisation. 

\begin{lemma}\label{cl1}
Almost surely, we have
\begin{equation}\label{eq:max-phi_m-discrete}
\max_\lambda \phi_{map}(\lambda,c_0) = \max_{i \leq n^{10}} \phi_{map}(\lambda_i,c_0) + O({\red n^{-7}}).
\end{equation}
\end{lemma}
\begin{proof}
This follows from a standard argument: we have
\[
| \max_\lambda \phi_{map}(\lambda,c_0) - \max_{i \leq n^{10}} \phi_{map}(\lambda_i,c_0) | \leq \max_{i \leq  n^{10}} \max_{\lambda \in [\lambda_i,\lambda_{i+1}]} |\phi_{map}(\lambda,c_0) -  \phi_{map}(\lambda_i,c_0) |
\]
and for any $\lambda, \lambda'$
\[
|\phi_{map}(\lambda,c_0) -  \phi_{map}(\lambda',c_0) | \leq |\min_f \sum_{e=(v,f(v)) } (W_e + \lambda C_e) - \min_{f} \sum_{e=(v,f(v))}(W_e + \lambda' C_e)| + |\lambda - \lambda'|c_0.
\]
If we take $\tilde{f}$ to be an optimal mapping for $\lambda$ and $\tilde{f}'$ for $\lambda'$, we can conclude that
\[
\min_f \sum_{e=(v,f(v))} (W_e + \lambda C_e) \leq \sum_{e=(v,\tilde f'(v))} (W_e + \lambda C_e) \leq \min_{f'} \sum_{e=(v,f '(v))}(W_e + \lambda' C_e) + \max_f\sum_{e =(v, f (v))} |\lambda-\lambda'|C_e
\]
which easily gives (by estimating each $C_e$ by $1$ and exchanging the roles of $\lambda$ and $\lambda'$)
\[
|\min_T \sum_{e \in T} (W_e + \lambda C_e) - \min_{F'} \sum_{e \in F'}(W_e + \lambda' C_e)| \leq |\lambda-\lambda'|n.
\]
Since $c_0 = O(n)$ and $|\lambda - \lambda_i| \leq {\red n^{-8}}$, we finish the argument. 

\end{proof}
The function $\f_{map}(\l,c_0)$ is concave and will be strictly concave with probability one. Let $\l^*$ denote the value of $\l$ maximising $\f$ and let $\l^{**}$ be the closest discretised value to $\l^*$. Let $f^{**}$ be the mapping that minimises $W(f)+\l^{**}C(f)$. We will see in the following that 
\beq{lowerl}{
\l^*\geq \frac{1}{n^2}\text{ w.h.p.}
}
{\red The asymptotic values of $\l^*$ are derived in Section \ref{nopt}.}
\begin{lemma}\label{hamm}
Assuming \eqref{lowerl}, then
\[
f^*=f^{**}\text{ w.h.p.}
\]
\end{lemma}
\begin{proof}
{\red From the above, we have to show that 
\beq{arg}{
\mathrm{argmin}_j\set{W_{i,j}+\l^*C_{i,j}}=\mathrm{argmin}_j\set{W_{i,j}+\l^{**}C_{i,j}}\text{ for }i\in [n].
}
Now $|\l^{**}-\l^*|\leq {\red n^{-8}}$ and so $|Z_{i,j}^*-Z_{i,j}^{**}|\leq n^{-8}$ where $Z_{i,j}^*=W_{i,j}+\l^*C_{i,j}$ etc.  Furthermore, if $U, V, X,Y$ are independent uniform $[0,1]$ random variables then $\Pr(|X-Y+\l(U-V)| \leq \delta)\leq \delta/\l$ for any choice of $\delta,\l > 0$. Thus
\beq{arg1}{
\Pr\brac{\exists i,j_1,j_2: j_1 \neq j_2, \ |Z^*_{i,{\red j_1}} - Z^{*}_{i,j_2}| \leq \frac{2}{n^8}}
\leq n^3\cdot\bfrac{2}{\l^*{\red n^8}}=o(1),
}
under the assumption that $(\l^*)n^5\to\infty$. 

Now switching from $\l^*$ to $\l^{**}$ can only change a $Z_{i,j}$ by at most $n^{-8}$ and \eqref{arg1} implies that this is not enough to affect which is the smallest $Z_{i,j}^{**}$, for a fixed $i$.
}
\end{proof}
\subsection{Cycles}\label{cycles}
A mapping $f$ gives rise to a digraph $D_F=([n],\set{(v,f(v)):v\in [n]}$. The digraph $D_F$ splits into components consisting of directed cycles plus arborescences attached to these cycles.
\begin{lemma}\label{lm:cyc}
There is a universal constant $K$ such that a uniform random mapping $F: [n]\to [n]$ has at most $K\log n$ cycles with probability at least $1 - O(n^{-50})$.
\end{lemma}
\begin{proof}
If we condition on the set $C$ of vertices on cycles, then the cycles define a random permutation of the elements of $C$. One can see this by observing that if we remove the edges from these cycles and replace them with another collection of cycles that cover $C$ then we get another digraph of a mapping. This explains that each set of cycles that covers $C$ has the same set of extensions to a mapping digraph i.e. arises in the same number of mappings.

Let $C=[m]$. Let $\pi$ be a random permutation of $[m]$. Let $X$ denote the size of the cycle containing 1. Then
\[
\p{X=i}=\frac{(m-1)(m-2)\cdots(m-i+1)\times (m-i)!}{m!}=\frac{1}{m}.
\]
{\bf Explanation:} The factor $(m-1)(m-2)\cdots(m-i+1)$ is the number of ways of completing the cycle containing 1 and $(m-i)!$ is the number of ways of computing the vertices not on $C$.

Now let $Y$ denote the number of cycles in $\pi$. From this we can argue that 
\[
\p{Y\geq t}\leq \p{Bin(t,1/2)\leq \lceil \log_2m\rceil}.
\]
{\bf Explanation:} We flip a sequence of fair coins. If we get a head in the first one, then we interpret this as vertex 1 being on a cycle $C_1$ of size at least $m/2$ and then we continue the experiment with $[m]\setminus C_1$. If we get a tail, then we continue the experiment with $[m]$.

So, by the Chernoff bounds, if $Z$ is the number of cycles in a random mapping, then for $K\geq 2$,
\begin{align*}
&\p{Z\geq K\log_2n)}\leq \p{Bin(K\log_2n,1/2)\leq \lceil\log_2n\rceil} \\
&\qquad\qquad\leq \exp\left\{-\frac{(K-2)^2}{2K^2}\cdot A\log_2n\right\} =n^{-(K-2)^2/2K}.
\end{align*}
\end{proof}
\section{Proof of Theorem \ref{thm:main-res}}
It will be convenient to first argue about the cost of an optimal mapping and then amend it to obtain an \emph{almost} optimal arborescence with the (asymptotically) correct cost. Namely, we define $W_{map}^*(c_0)$ to be the optimal value of the integer program $IP_{map}$ of Section \ref{Dual}.

First, we show that with high probability
\begin{equation}\label{eq:mapping}
W_{map}^*(c_0) \approx \begin{cases}\frac{\pi n}{8c_0}.&\text{Case 1.}\\f(\b)-\a\b\text{ where }f'(\b)=\a&\text{Case 2}.\\(g(\b)-\a\b)n\text{ where }g'(\b)=\a&\text{Case 3}.\end{cases}
\end{equation}
and then we modify an \emph{almost} optimal mapping (with the slightly more restricted budget $c_0-\delta$ for the cost) to obtain an arborescence $A$ which with high probability will satisfy $W(A) \approx W_{map}^*(c_0)$ as well as the cost constraint $C(A) = \sum_{e \in A} C_e \leq c_0$. Since 
\[
W^*_{arb}(c_0) \geq W_{map}^*(c_0) \approx W(A) \geq W_{arb}^*(c_0),
\] 
this will show that $W_{arb}^*(c_0) \approx \frac{\pi n}{8c_0}$ in Case 1., etc., as desired.
\subsection{A near optimal mapping}\label{nopt}
Our goal is to show \eqref{eq:mapping}. By weak duality or the fact that $LP_{map}$ relaxes $IP_{map}$ we have
\begin{equation}\label{eq:mapping-dual-low-bd}
W_{map}^*(c_0) \geq \max_{\lambda} \phi_{map}(\lambda,c_0).
\end{equation}

To handle $\phi_{map}$, note that the minimum  over the mappings is of course attained by choosing the best edge for each vertex, that is
\begin{equation}\label{eq:phi_m}
\phi_{map}(\lambda, c_0) = \sum_{i \leq n} \min_{j\neq i} \{W_{(i,j)} + \lambda C_{(i,j)}\} - \lambda c_0.
\end{equation}
Now the analysis splits into three cases according to the value of $c_0$.

\bigskip
{\bf Case 1:} $c_0 \in [\om,n/\om]$.

First we take the maximum over $i$ {\red (the index for the discretization)}. The function $(1+o(1))\sqrt{\frac{\pi}{2}}\sqrt{\lambda n} - \lambda c_0$ is strictly concave and has a global maximum at $\lambda^* = (1+o(1))\frac{\pi n}{8c_0^2}$, satisfying \eqref{lowerl}. Note that with our assumption on $c_0$, this value of $\l$ is in the third {\red or the fourth} range of Lemma \ref{lm:Emin}. 

By \eqref{eq:Emin-lam-small} and the concentration result of Lemma \ref{lm:conc} applied to $\varepsilon = n^{-1/5}$, we have 
\begin{lemma}\label{claim1}
\[
\phi_{map}(\lambda_i,c_0) = (1+o(1))\sqrt{\frac{\pi}{2}}\sqrt{\lambda_i n} - \lambda_ic_0,
\]
for every $i \leq n^{10}$ with probability at least $1 - O(n^{-{\red 89}})$.
\end{lemma}
Thus the optimal value over $\lambda = \lambda_i$, $i \leq {\red n^{10}}$, is
\begin{align*}
\max_{i \leq n^{10}} \phi_{map}(\lambda_i,c_0) &= (1+o(1))\sqrt{\frac{\pi}{2}}\sqrt{(\lambda^* +O(n^{-8}))n} - (\lambda^* + O({\red n^{-8}})) c_0 \\
&=(1+o(1))\frac{\pi}{8}\frac{n}{c_0}
\end{align*}
which together with {\red Lemma \ref{cl1}} gives that with probability at least ${\red 1-O(n^{-90})}$
\begin{equation}\label{eq:mapping-dual-answer}
\max_\lambda \phi_{map}(\lambda,c_0) = (1+o(1))\frac{\pi}{8}\frac{n}{c_0} + O(n^{-3}) = (1+o(1))\frac{\pi}{8}\frac{n}{c_0}.
\end{equation}
The last step is to tighten the cost constraint a little bit and consider $c_0' = c_0-1$. {\red By using \eqref{eq:mapping-dual-answer} twice and recalling \eqref{eq:mapping-dual-low-bd}, we obtain}
\mult{eq:chain}{
W_{map}^*(c_0) \geq \max_\lambda \phi_{map}(\lambda,c_0) = (1+o(1))\frac{\pi}{8}\frac{n}{c_0} = (1+o(1))\frac{\pi}{8}\frac{n}{c_0'} \\
= (1+o(1))\max_\lambda \phi_{map}(\lambda,c_0') \geq W(f^*)-W_{\max},
}
where $f^*=f^*(c_0')$ is as in \eqref{clearly} and 
\beq{Cf}{
C(f^*) \leq c_0' + C_{\max}(f^*)\leq c_0'+1 {\red =} c_0.
}
 This means that the solution $f^*$ is feasible and thus $W(f^*) \geq W_{map}^*(c_0)$. We have from Corollary \ref{cor1} and our expressions for the optimal value of $\l$ that
\[
W_{\max}=O\brac{1+\frac{n}{c_0^2}}\sqrt{\log n/n}=o\bfrac{n}{c_0}=o(W(f^*)).
\]
 Going back to \eqref{eq:chain} we see that $W_{map}^*(c_0) \approx\frac{\pi}{8}\frac{n}{c_0}$, thus showing \eqref{eq:mapping} holds with probability at least $1-O(n^{-90})$. Moreover,
\begin{equation}\label{eq:Wmap-dual}\
W_{map}^*(c_0) \approx \max_\lambda \phi_{map}(\lambda,c_0).
\end{equation}
{\bf Case 2:} $c_0 = \alpha n$, $\alpha = O(1)$.

If $\a>1/2$ then w.h.p. we can take the mapping $f(v)$ where $W_{(v,f(v))}=\min\set{W_{(v,w)}:w\neq v}$. Then the sum $\sum_vC_{(v,f(v))}$ being the sum of $n$ independent uniform $[0,1]$ random variables is asymptotically equall to $n/2$ w.h.p. This implies that $f$ defines a feasible mapping w.h.p.

Assume then that $\a<1/2$. We use the argument of Case 1 and we omit details common to both cases. We first check that the optimal value $\l^*$ is in the second range of Lemma \ref{lm:Emin}, {\red justifying \eqref{lowerl} in this case}. To see this observe that if $\l=\frac{\b}{n}$ where $\b\in \left[\frac{1}{\log n},\log n\right]$ then $\f_{map}(\l,c_0) \approx f(\b)-\a\b$. Now Lemma \ref{f(b)} affirms that $f(\b)-\a\b$ is concave and that there is a unique positive solution $\b^*$ to $f'(\b)=\a$. It follows that $\max_\l\f_{map}(\l,c_0)\approx f(\b^*)-\a\b^*$. 

We let $c_0'=c_0-1\approx c_0$. Using the continuity of $f$ and $W_{\max}=o(1)$ from \eqref{Wmax}, we have $W^*_{map}(c_0)\ge (1+o(1))W(f^*)$ in \eqref{eq:chain} and by \eqref{clearly} we have $C(f^*)\leq c_0'+1=c_0$.
Again, \eqref{lowerl} is satisfied.

\bigskip
{\bf Case 3:} $c_0 = \alpha$, $\alpha = O(1)$.

If $\a<1$ then w.h.p. the problem is infeasible. This is because the sum $S=\sum_v\min_{w}W_{(v,w)}$ is the sum of $n$  i.i.d.  random variables and this sum has mean $\frac{n}{n+1}$ and Lemma \ref{lm:conc} with $\l=0$ shows that $S$ is concentrated around its mean.

Assume then that $\a>1$. We use the argument of Case 1 and as in Case 2, we omit details common to both cases. We first check that the optimal value $\l^*$ is in the {\red fourth} range of Lemma \ref{lm:Emin}. To see this observe that if $\l=\b n$ where $\b\in \left[\frac{1}{\log n},\log n\right]$ then $\f_{map}(\l,c_0)\approx n(g(\b)-\a\b)$. Now Lemma \ref{g(b)} affirms that $g(\b)-\a\b$ is concave and that there is a unique positive solution $\b^*$ to $g'(\b)=\a$. It follows that $\max_\l\f_{map}(\l,c_0)\approx n(g(\b^*)-\a\b^*)$. It only remains to check that $C_{\max}(f^*)=o(1)$ so that we can apply \eqref{clearly}. Again, \eqref{lowerl} is satisfied.

We now let $c_0'=c_0-1/n^{1/2}\approx c_0$. Using the continuity of $g$ and $W_{\max}\leq 1$ we have $W^*_{map}\ge (1+o(1))W(f^*)$ in \eqref{eq:chain} and we have $C(f^*)\leq c_0'+K\frac{\log n}{n} \leq c_0$.

One final point. Our expressions for $\f(\l)$ are only valid within a certain range. But because, $\f$ is concave and we have a vanishing derivative, we know that the values outside the range cannot be maximal.

\subsection{From a mapping to an arborescence}\label{From}
{\bf Case 1:}\\
Fix $c_0$ and let $c_0' = c_0(1 - \varepsilon)$ with $\varepsilon = n^{-1/4}\log n$. Since $c_0' \approx c_0$, by \eqref{eq:mapping} and \eqref{eq:Wmap-dual}, we have
\[
W^*_{arb}(c_0) \geq W^*_{map}(c_0) \approx \frac{\pi n}{8c_0} \approx \frac{\pi n}{8c_0'} \approx  W^*_{map}(c_0') \approx \max_\lambda \phi_{map}(\lambda,c_0').
\]
Let the maximum on the right hand side be attained at some $\lambda^*$ and let $\l^{**}$ be the closest discretized value. Let $f^*$ be as defined in Section \ref{Dual} and $f^{**}$ minimise $W(f)+\l^{**}C(f)$. Then, we have from {\red Lemma \ref{cl1}} and \eqref{clearly} that 
\begin{equation}\label{eq:W-C}
\begin{split}
W(f^{*})\leq  W^*_{map}(c_0)+W_{\max}+O(n^{-3})\\
C(f^{*})\leq c_0'+C_{\max}.
\end{split}
\end{equation}
We now argue that with high probability it is possible to modify $f^{*}$ to obtain a feasible arborescence $A$, that is of cost at most $c_0$, having weight very close to $W^*_{map}$. 

By Lemmas \ref{lm:opt-is-unif} and \ref{lm:cyc}, with probability at least $1-O(n^{-10})$, $f^{*}$ has at most $K\log n$ cycles for some universal constant $K$. Then the largest component {\red of the digraph $D_{f^*}$}, call it $U$, has at least $\frac{n}{K\log n}$ vertices. 

{\red 
We know from Corollary \ref{cor1} that w.h.p. we do not use any edge of weight more than $O((1+\l^*)\sqrt{\log n/n})$ in constructing $f^*$. It follows, as in Karp and Steele \cite{KS}, that given $f^*$, we can treat the edges of weight at least $\hW=K(1+\l^*)\sqrt{\log n/n}$ as being independent samples from $[\hW,1]$. Similarly, edges of cost at least $\hC=K(1+1/\l^*)\sqrt{\log n/n}$ can be treated as independent samples from $[\hC,1]$. 
}

We consider two cases:

{\bf Case 1a: $n^{1/2}\leq c_0\leq n/\om$:}\\
For each cycle, choose arbitrarily one vertex belonging to it, say $v$, remove its out-edge, breaking the cycle and put instead the minimum weight out-edge connecting it to {\red the maximum component $U^*$}. This way $f^{*}$ is transformed into an arborescence, call it $A$. The probability that the added out edge weighs more than $2\hW$ is at most 
\[
(1-2\hW)^{\Omega(n/\log n)}\leq \exp\set{-\Omega\brac{\frac{n}{\log n}\cdot  \brac{1+\frac{n}{c_0^2}}\sqrt{\bfrac{\log n}{n}}}}=o(n^{-10}).
\]
We have $W^*_{map}=\Omega(n/c_0)=\Omega(\om)$ and \eqref{eq:W-C} and $W_{\max},C_{\max}\leq 1$. 
\begin{align*}
W(A)&\leq {\red W(f^{*})+2K\hW\log n=\brac{1+o(1)}W^*_{map}}.\\
C(A)&\leq c_0'+1+O(n^{-3})+K\log n\leq c_0.
\end{align*}
{\red To justify the final estimate for $W(A)$ we have used $\frac{\l^*\sqrt{\log n/n}}{n/c_0} =O\bfrac{\sqrt{\log n/n}}{c_0}=o(1)$}.

{\bf Case 1b: $\om\leq c_0\leq n^{1/2}$:}\\
%It follows from $\l^{*}=\Theta(n/c_0^2)$ that $\l^{*}=\Omega(1)$. It then follows from \eqref{eq:conc-Cmax} and Lemma \ref{hamm} that $C_{\max}(f^{*})=O(\sqrt{\log n/n})$. If therefore we delete every edge $e$ for which $C_e\geq n^{-1/4}$ from $\vec{K_n}$ and compute an optimal mapping, then w.h.p. we will get the same mapping $f^{*}$ as without doing the deletion. Now w.h.p., for any vertex $v$, there are at most $2n^{1/4}$ edges $e$ incident with  $C_e< n^{-1/4}$. 

%Now put back every edge that was deleted and consider the conditional distribution of $C_e$ of a deleted edge $e$. The distribution of $C_e$ will be uniform $[0,1]$, conditional on $C_e\geq n^{-1/4}$ and this is uniform $[n^{-1/4},1]$. Applying the same transformation from mapping to arborescence as in Case 1, but doing this as cheaply as possible, we see that
%\multstar{
%\p{\not\exists \ \text{out-edge $e\in E(v:U)$ such that } C_e \in[n^{-1/4}, 2n^{-1/4}]}\\ 
%\leq\left(1 - \frac{n^{-1/4}}{1-n^{-1/4}}\right)^{n/K\log n-2n^{1/4}} \leq e^{-\Omega(n^{3/4}/\log %n)}.
%}
%Taking the union bound over the cycles, we see that with high probability for each cycle there is a choice of an edge with $ W_e\leq 1,C_e \leq 2n^{-1/4}$. Thus, the difference of weight between $f^{*}$ and $A$ is at most $2K\log n$ and the difference of cost is at most $2K\log n \times n^{-1/4}$. Consequently, $C(A) \leq c_0(1 - \varepsilon) + 2Kn^{-1/4}\log n \leq c_0$ and therefore $A$ is feasible and we get
%\[
%W^*_{arb}(c_0) \leq W(A) \leq \frac{\pi n}{8c_0} + 2K\log n\approx \frac{\pi n}{8c_0}.
%\]
For each cycle, choose arbitrarily one vertex belonging to it, say $v$, remove its out-edge, breaking the cycle and put instead the minimum cost out-edge connecting it to $U^*$. This way $f^{*}$ is transformed into an arborescence, call it $A$. The probability that the added out edge costs more than $2\hC$ is at most 
\[
(1-2\hC)^{\Omega(n/\log n)}\leq \exp\set{-\Omega\brac{\frac{n}{\log n}\cdot  \brac{1+\frac{c_0^2}{n}}\sqrt{\bfrac{\log n}{n}}}}=o(n^{-10}).
\]
We have $W^*_{map}=\Omega(n/c_0)=\Omega(n^{1/2})$ and \eqref{eq:W-C} and $W_{\max},C_{\max}\leq 1$. 
\begin{align*}
W(A)&\leq {\red W(f^{*})+K\log n=\brac{1+o(1)}W^*_{map}}.\\
C(A)&\leq c_0'+1+O(n^{-3})+2K\hC\log n\leq c_0.
\end{align*}
{\red To justify the final estimate for $C(A)$ we have used $\frac{(\l^*)^{-1}\sqrt{\log n/n}}{c_0} =O\bfrac{c_0\sqrt{\log n/n}}{n}=o(1)$}.
This finishes the proof of Case 1.

{\bf Case 2:}\\
We have $c_0=\Omega(n)$ here and $\l^{**}=\b^{**}=\Theta(1)$.  We can therefore use \eqref{eq:conc-Cmax} to argue that w.h.p. $\max\set{W_{\max}(f^{*}),C_{\max}(f^{*})}=O(\sqrt{\log n/n})$.  We then can proceed as in Case 1b and use edges $e$ such that $W_e,C_e\in [n^{-1/4},2n^{-1/4}]$ to transform $f^{*}$ into an arborescence and w.h.p. change weight and cost by $o(1)$ only. 

{\bf Case 3:}\\
We have $\l^{**}=\b^{**}n=\Theta(n)$.  We can therefore use \eqref{eq:conc-Cmax} to argue that w.h.p. $C_{\max}(f^{*})=O(\sqrt{\log n/n})$.  We proceed as in Case 1b and use edges $e$ such that $W_e\leq 1,C_e\in [n^{-1/4},2n^{-1/4}]$ to transform $f^*$ into an arborescence. The extra cost in going from mapping $f^{*}$ to an arborescence $A$ is $O(n^{-1/4}\log n)=o(1)$, {\red thus $C(A) \leq c_0' + r_n = c_0(1-\e ) + r_n$, where $r_n = O(\sqrt{\log n/n}) + O(n^{-1/4}\log n)$, so $C(A) < c_0$ provided that $\e$ is chosen such that $\e c_0 > r_n$.}  The extra weight is $O(\log n)$ which is much smaller than the optimal weight which is $\Omega(n)$ w.h.p.
\section{More general weights and costs}\label{a<1}
We now consider the case where we have $W_e,C_e,e\in E(K_n)$ distributed as independent copies of $U^s,s<1$, {\blue $U \sim \text{Unif}([0,1])$}. We follow the same ideas as for $s=1$, but there are technical difficulties. Let us first though explain the need for the lower bound on $c_0$ in Theorem \ref{thm:main-resx}, up to a logarithmic factor.
\begin{lemma}\label{lowerbound}
Let $X_1,X_2,\ldots,X_n$ be independent copies of $U^s$ and let $Y=min_{i\leq n}X_i$. Then 
\beq{minX}{
\E Y\approx \G({\blue s +1})n^{-s}.
}
\end{lemma}
\begin{proof}
\begin{align*}
\E\min_{i\leq n}X_i &=\int_{t=0}^1 \Pr({\blue U}>t^{1/s})^n \dd t\\
&=\int_{t=0}^1(1-t^{1/s})^n\dd t\\
&=s\int_{t=0}^1(1-s)^ns^{s-1}\dd s\\
&=s B(n+1,s)\qquad \text{ Beta distribution}\\
&=\frac{\G(n+1)\G(s+1)}{\G(n+s+1)}\\
&\approx \G(s+1)\frac{(n/e)^n}{((n+s)/e)^{n+s}}\qquad \text{Stirling's approximation}\\ 
&= \frac{\G(s+1)e^s}{(n+s)^s} \bfrac{n}{n+s}^n\\
&\approx \frac{\G(s+1)}{n^s}.
\end{align*}
\end{proof}
It follows from \eqref{minX} that the expected weight of a minimum random mapping is asymptotically equal to $\G(s+1)n^{1-s}$. This being the expectation of the sum of $n$ independent bounded random variables distributed as $Y$ in Lemma \ref{lowerbound}, we see from Hoeffding's theorem \cite{H} that it is concentrated around its mean. This explains the relevance of the lower bound in \eqref{c0range}, up to the $\log n$ factor.

Our next task is get a version of Lemma \ref{lm:Emin} for our more general random variables.
\begin{lemma}\label{lm:Emin1}
Let $X_1,X_2,\ldots, X_n$ and $Y_1,Y_2,\ldots, Y_n$ be independent copies of $U^s$. Suppose that 
\beq{valambda}{
%n\gg\max\set{\l,\l^{-1}}^{1/s}\log n.
\left(\frac{\log n}{n}\right)^s \ll \l \ll \left(\frac{n}{\log n}\right)^s. 
}
 Let
\[
C_s={\blue \G(s/2+1)}\bfrac{\G(2/s+1)}{\G(1/s+1)^2}^{s/2}.
\]
Then
\[
\E\min_{i \leq n}\{X_i + \lambda Y_i\} \approx\frac{C_s \l^{1/2}}{n^{s/2}}.
\]
\end{lemma}
\begin{proof}
Let $X,Y$ be independent copies of $U$. After some elementary computations we see the following.

{\bf Case 1: $\l>1$.}
\beq{Genmin}{
\p{X^s+\l Y^s\leq t}=\begin{cases}
                                           \int_{x=0}^{t^{1/s}}\bfrac{t-x^{s}}{\l}^{1/s}\dd x&0<t<1.\\
                                           \int_{x=0}^{1}\bfrac{t-x^{s}}{\l}^{1/s}\dd x&1<t<\l.\\
                                           (t-\l)^{1/s}+ \int_{x=(t-\l)^{1/s}}^{1}\bfrac{t-x^{s}}{\l}^{1/s}\dd x&\l<t<1+\l\\
                                           1&1+\l<t.
                                          \end{cases}
}
Only the first integral in the above seems computable at the moment and this restricts our range for $c_0$. We have
\beq{I1}{
I_1(t)=\int_{x=0}^{t^{1/s}}\bfrac{t-x^{s}}{\l}^{1/s}\dd x=\frac{t^{2/s}}{s\l^{1/s}} \int_{y=0}^1y^{1/s-1}(1-y)^{1/s}\dd y=\frac{t^{2/s}}{\l^{1/s}}\frac{\G(1/s+1)^2}{\G(2/s+1)}.
}
For the remaining integrals, we only have lower bounds. Let $Z=\min_{i\leq n}\set{X_i+\l Y_i}$. Going back to \eqref{tti} we see that we need lower bounds to show that the contributions of these integrals to $\E(Z)$ is negligible. We note also that we will only be concerned with small values of $t$ in the sequel and so we need to show that these integrals are large compared to the R.H.S. of \eqref{I1}.
%%%These previous estimates (if correct ;)) would need some explanations
%\begin{align*}
%I_2=\int_{x=0}^{1}\bfrac{t-x^{s}}{\l}^{1/s}\dd x&=\frac{t^{1/s}}{\l^{1/s}}\int_{x=0}^1\brac{1-\frac{x^{\blue s}}{t}}^{1/s}\geq\frac{1}{\l^{1/s}2^{1/s+1}}. \\
%I_3=(t-\l)^{1/s}+ \int_{x=(t-\l)^{1/s}}^{1}\bfrac{t-x^{s}}{\l}^{1/s}&\geq \frac{1}{2^{1/s+1}}.
%\end{align*}
{\blue 
Put $D_s=\frac{\G(1/s+1)^2}{\G(2/s+1)}$.
Crudely, by the monotonicity in $t$ of $\p{X^s+\l Y^s\leq t}$,
\[
I_3 \geq I_2 \geq I_1(1) = \frac{1}{\l^{1/s}}D_s.
\]
Then 
\mult{ExpZ}{
%\E(Z)=\int_{t=0}^1 (1-I_1)^n\dd t+\int_{t=1}^\l (1-I_2)^n\dd t+\int_{t=\l}^{1+\l}(1-I_3)^n\dd t =\\
%\int_{t=0}^1\brac{1-\frac{t^{2/s}}{\l^{1/s}}\frac{\G(1/s+1)^2}{\G(2/s+1)}}^n\dd t+O\brac{e^{-\Omega(n/\l^{1/s})}}.
\E(Z)=\int_{t=0}^1 (1-I_1)^n\dd t+\int_{t=1}^\l (1-I_2)^n\dd t+\int_{t=\l}^{1+\l}(1-I_3)^n\dd t =\\
\int_{t=0}^1\brac{1-\frac{t^{2/s}}{\l^{1/s}}D_s}^n\dd t+O\brac{\l e^{-D_s n/\l^{1/s}}}.
}
}
We substitute $t=\bfrac{s\l^{1/s}}{D_s n}^{s/2}$ and estimate
\begin{align*}
\int_{t=0}^1\brac{1-\frac{D_s t^{2/s}}{\l^{1/s}}}^n\dd t&\leq \int_{t=0}^{\infty}\exp\set{-\frac{D_s nt^{2/s}}{\l^{1/s}}}\dd t\\
&= \frac{s\l^{1/2}}{2(D_s n)^{s/2}}\int_{t=0}^{\infty}e^{-s}s^{s/2-1}\dd s\\
&=\frac{\G(s/2+1)\G(2/s+1)^{s/2}\l^{1/2}}{\G(1/s+1)^s n^{s/2}}.
\end{align*}
On the other hand, if $t_0=\frac{\l^{1/2}}{n^{s/4}\log n}$ then \\
{\red Is $t_0\leq 1$?}
\begin{align*}
\int_{t=0}^1\brac{1-\frac{D_s t^{2/s}}{\l^{1/s}}}^n\dd t&\geq \int_{t=0}^{t_0}\brac{1-\frac{D_s t^{2/s}}{\l^{1/s}}}^n\dd t\\
&=\int_{t=0}^{\blue t_0}\exp\set{-\frac{D_s nt^{2/s}}{\l^{1/s}}+O\bfrac{nt_0^{4/s}}{\l^{2/s}}}\dd t\\
&\approx \int_{t=0}^{t_0}\exp\set{-\frac{D_s nt^{2/s}}{\l^{1/s}}}\dd t\\
&=\int_{t=0}^{\infty}\exp\set{-\frac{D_s nt^{2/s}}{\l^{1/s}}}\dd t-O\brac{\exp\set{-\frac{D_s nt_0^{2/s}}{\l^{1/s}}}}\\
&=\int_{t=0}^{\infty}\exp\set{-\frac{D_s nt^{2/s}}{\l^{1/s}}}\dd t- O\brac{\exp\set{-\frac{D_s n^{1/2}}{\log^{2/s}n}}}\\
&\approx \frac{\G(s/2+1)\G(2/s+1)^{s/2}\l^{1/2}}{\G(1/s+1)^s n^{s/2}}.
\end{align*}
{\blue It remains to notice that thanks to the assumption on $\l$, the error term $O\brac{\l e^{-D_s n/\l^{1/s}}}$ is small relative to the main term.}

{\bf Case 2: $\l\leq 1$.} We have
\[
\E\min_{i \leq n}\{X_i + \lambda Y_i\} = \l\E\min_{i \leq n}\{\l^{-1}X_i + Y_i\} \approx \frac{C_s\l^{1/2}}{n^{s/2}}.
\]
\end{proof}
\begin{corollary}\label{cor:Emin-low-bd1}
Under the assumptions of Lemma \ref{lm:Emin1}, we have
\beq{corbd1}{
n\E \min_{i \leq n}\{X_i + \lambda Y_i\} \approx C_s\l^{1/2}n^{1-s/2}.
}
\end{corollary}
Our next task is to prove an appropriate version of Lemma \ref{concsec}.
\begin{lemma}\label{lm:conc1}
Let $W_{(i,j)}$ and $C_{(i,j)}$, $i, j \leq n$, be i.i.d. copies of $U^s$ and suppose that \eqref{valambda} holds. For $X_i = \min_{j} \{W_{(i,j)} + \lambda C_{(i,j)}\}$, $S = \sum_{i \leq n} X_i$ and $\varepsilon = \Omega(n^{-1/5})$, we have
\begin{equation}\label{eq:concx}
\p{|S - \E S| > \varepsilon \E S} = O(n^{-99}).
\end{equation}
Moreover,
\begin{equation}\label{eq:conc-Cmax1}
\p{\exists i:\; X_i > 10(1+\lambda)n^{-s/2}\log^{s/2}n} \leq n^{-99}.
\end{equation}
\end{lemma}
\begin{proof}
We closely follow the argument of Lemma \ref{concsec}, making adjustments as necessary. Let $M = 10(1+\lambda)n^{-s/2}\log^{s/2}n$ and $B$ be the event that for some $i$, $X_i \geq M$. We have,
\begin{equation}\label{eq:conc-two-termsa}
\p{|S - \E S| > \varepsilon \E S} \leq \p{B} + \p{(|S - \E S| > \varepsilon \E S)\wedge B^c}.  
\end{equation}
First we bound $\p{B}$. By the union bound and independence,
\[
\p{B} \leq n\p{X_1 \geq M} = n\left[\p{W_{(1,1)}+\lambda C_{(1,1)} \geq M}\right]^{n}.
\]
We use $W_{(1,1)}+\lambda C_{(1,1)} \leq (1+\lambda)\max\{W_{(1,1)},C_{(1,1)}\}$ and note that since these variables are distributed as $U^s$, we have $\p{\max\{W_{(1,1)},C_{(1,1)}\} \geq u} = 1 - u^{2/s}$ for $u < 1$. We thus get
\[
\p{B} \leq n\left[1 - 10^{2/s}\frac{\log n}{n}\right]^{n} \leq ne^{-100\log n} = n^{-99},
\]
which establishes \eqref{eq:conc-Cmax1}.

The second term in \eqref{eq:conc-two-termsa} can be bounded using Chernoff's inequality because on $B^c$, $X_i = X_i\1_{X_i \leq M}$, that is $S$ can be treated as a sum of $n$ independent random variables $\tilde X_i = X_i\1_{X_i \leq M}$ with $\tilde X_i \in [0,M]$. Clearly $\tilde X_i \leq X_i$ and $\tilde S = \sum \tilde X_i \leq S$, so
\begin{align*}
\p{(|S - \E S| > \varepsilon \E S)\wedge B^c} = \p{(|\tilde S - \E S| > \varepsilon \E S)\wedge B^c}  \leq \p{|\tilde S - \E S| > \varepsilon \E S}.
\end{align*}
By the Chernoff bound
\[
\p{|\tilde S - \E \tilde S| > \varepsilon \E \tilde S} \leq 2\exp\left\{-\frac{\varepsilon^2\E\tilde S}{3M}\right\}.
\]
Note that
\[
|\tilde S - \E S| \leq |\tilde S - \E \tilde S| + |\E S - \E \tilde S|.
\]
and
\begin{align*}
 |\E S - \E \tilde S| = \left|\E \sum X_i\1_{X_i > M} \right| \leq (1+\lambda)\E\sum \1_{X_i > M} &\leq (1+\lambda)n\p{X_1  > M} \\
 &= O(n^{-90}),
\end{align*}
thanks to \eqref{eq:conc-Cmax1}. Moreover, by Corollary \ref{cor:Emin-low-bd1}, $\E S=\Omega(n^{1-s/2}\l^{1/2})$ thus
\[
|\E S - \E \tilde S| \leq \frac{1}{2}\varepsilon \E S
\]
and we get
\begin{align*}
\p{|\tilde S - \E S| > \varepsilon \E S} \leq  \p{|\tilde S - \E \tilde S| > \frac{1}{2}\varepsilon \E S} &\leq \p{|\tilde S - \E \tilde S| > \frac{1}{2}\varepsilon \E \tilde S} \\
&\leq 2\exp\left\{-\frac{\varepsilon^2\E\tilde S}{12M}\right\}.
\end{align*}
Finally, observe that {\blue since $\frac{\l^{1/2}}{1+\l} \geq \frac{1}{2}\max\{\l,\l^{-1}\}^{1/2} \gg \frac{1}{2}\left(\frac{\log n}{n}\right)^{s/2}$}, 
\[
\frac{\E \tilde S}{M} \geq \frac{\E S}{2M}\approx \frac{C_s n\l^{1/2}}{20(1+\lambda)\log^{s/2}n} = {\blue \Omega\left(n^{1-s/2}\right)}.
\]
So
\[
%\begin{equation}\label{eq:conc-B}
\p{|S - \E S| > \varepsilon \E S, B^c} \leq 2e^{-{\blue \Omega(n^{9/10-s/2})}}.
\]
%\end{equation}
In view of \eqref{eq:conc-two-termsa}, this combined with \eqref{eq:conc-Cmax1} finishes the proof of \eqref{eq:concx}.
\end{proof}
In place of Corollary \ref{cor1} we have
\begin{corollary}\label{cor1x}
Let $M_n$ denote the minimum  weight of a mapping with weights $W_e+\l C_e,e\in E(\vec{K}_n)$, with $\l$ as in Lemma \ref{lm:conc1}. Then with probability $1-O(n^{-90})$,
\[
M_n\approx C_s n^{1-s/2}\l^{1/2}\text{ and }W_{\max}=O((1+\l)n^{-s/2}\log^{s/2}n)=o(M_n).
\]
\end{corollary}
\begin{proof}
This follows from Lemmas \ref{lm:Emin1} and \ref{lm:conc1}.
\end{proof}
The results of Sections \ref{Dual}, \ref{discrete} and \ref{cycles} carry over. We now establish
\begin{lemma}
Assuming \eqref{c0range}, we have that w.h.p.,
\beq{maxaphi1}{
\max_\l\f_{map}(\l,c_0)\approx \frac{C_s^2n^{2-s}}{4c_0}.
}
\end{lemma}
\begin{proof}
We have
\beq{c0}{
%n\gg\max\set{\l,\l^{-1}}^{1/s}\log n
{\blue \left(\frac{\log n}{n}\right)^s \ll \l \ll \left(\frac{n}{\log n}\right)^s}
\text{ and }  n^{1-s}\log n\ll c_0\ll \frac{n}{\log^{1/2}n}.
}
Now consider the function
\beq{fun1}{
C_s n^{1-s/2} \l^{1/2}-c_0\l.
}
This is a concave function of $\l$ and it will be maximised when
\beq{valla}{
C_s n^{1-s/2}\l^{-1/2}/2=c_0\text{ or }\l=\frac{C_s^2n^{2-s}}{4c_0^2},
}
and we note that this value of $\l$ is consistent with \eqref{c0}.

Substituting for $\l$ in \eqref{fun1} and simplifying, we see that \eqref{maxaphi1} holds.
\end{proof}
Equations \eqref{c0range} and  \eqref{maxaphi1} imply that w.h.p. $\max_\l\f_{map}(\l,c_0)\gg n^{1-s}$. So, we can as in Section \ref{nopt} argue with $c_0$ replaced by $c_0'=c_0-1$ and deduce from \eqref{Cf} that 
\[
W^*_{map}(c_0)\approx  \frac{C_s^2n^{2-s}}{4c_0}\ w.h.p.
\]
Proceeding as in Section \ref{From} we see that w.h.p.
\[
W^*_{arb}(c_0)=W^*_{map}(c_0)+O(\log n)
\]
and this completes the proof of Theorem \ref{thm:main-resx}.
\subsection{A coupling argument}\label{mgd}
We follow an argument from Janson \cite{Jan}. We will asssume that $W_e,C_e$ have the distribution function $F_w(t)=\Pr(X\leq t)$, of a random variable $X$, that satisfies $F(t)\approx at^{1/s},s\leq 1$ as $t\to 0$. The constant $a>0$ can be dealt with by scaling and so we assume that $a=1$ here. For a fixed edge and say, $W_e$, we consider random variables $W_e^<,W_e^>$ such that $W_e^<$ is distributed as $U^{s+\e_n}$ and $W_e^>$ is distributed as $U^{s-\e_n}$, where $\e_n=1/10\log n$. (This choice of $\e_n$ means that $n^{s+\e_n}=e^{1/10}n^s$.)  Then suppose that $X$ has the distribution $F^{-1}(U)$. We couple $X,W^<,W^>$ by generating a copy $U_e$ of $U$ and then putting $W_e^<=F_<^{-1}(U_e)=\log\bfrac{1}{1-u}^{s-\e_n}$. $F_>$ is defined similarly. The coupling ensures that $W_e^<\leq W_e\leq W_e^>$ as long as $W_e\leq\e_n$.

Given the above set up, it only remains to show that w.h.p. edges of length $W_e>\e_n$ or cost $C_e>\e_n$ are not needed for the upper bounds proved above. We can ignore the lower bounds, because they only increase if we exclude long edges. But this follows from Corollary \ref{cor1x}.

\section{Conclusion}
We have determined the asymptotic optimum value to Problem \eqref{prob} w.h.p. The proof is constructive in that we can w.h.p. get an asymptotically optimal solution \eqref{prob} by computing arborescence $A$ of the previous section. When weights and costs are uniform $[0,1]$, our theorem covers almost all of the possibilities for $c_0$, although there are some small gaps between the 3 cases. Our results for more general distributions have a more limited range and further research is needed to extend this {\blue part} of the paper. We have also considered more general classes of random variable and here we have a more limited range for $c_0$. 

The present result assumes that cost and weight are independent. It would be more reasonable to assume some positive correlation. This could be the subject of future research. One could also consider more than one constraint.

\end{document}